\title{Limits of locally-globally convergent graph sequences}
\author{Hamed Hatami\\
School of Computer Science, McGill University\\
Montreal, Quebec, Canada H3A 0G4\thanks{Research of H.H.\ was
supported by an NSERC and an FQRNT grant. Research of L.L.\ was
supported by ERC Grant No.~227701 and OTKA grant No.~CNK 77780.
Research of B.Sz.\ was supported by an NSERC grant.}\\[.5cm]
L\'aszl\'o Lov\'asz \\
Institute of Mathematics, E\"otv\"os Lor\'and University
Budapest, Hungary H-1117\\[.5cm]
Bal\'azs Szegedy\\
Department of Mathematics, University of Toronto\\
Toronto, Ontario, Canada M5S2E4}
\date{December 2012}

\documentclass{article}
\usepackage{ntheorem}
\usepackage{color}
\usepackage{latexsym}
\usepackage[hypertex]{hyperref}
\usepackage{amsmath,amssymb,bbm,srcltx}

\newcommand{\ignore}[1]{}

\ignore{

\newtheorem{prelem}{{\bf Theorem}}

}

\newtheorem{theorem}{Theorem}[section]
\newtheorem{corollary}[theorem]{Corollary}
\newtheorem{lemma}[theorem]{Lemma}
\newtheorem{proposition}[theorem]{Proposition}
\newtheorem{claim}{Claim}
\theorembodyfont{\rmfamily}
\newtheorem{remark}[theorem]{Remark}
\newtheorem{example}[theorem]{Example}
\newtheorem{definition}[theorem]{Definition}
\newtheorem{conjecture}[theorem]{Conjecture}

\newtheorem{question}[theorem]{Question}

\newenvironment{proof}{{\bf Proof.}}{\hfill\rule{2mm}{2mm}\medskip}

\def\proofend{\hfill\rule{2mm}{2mm}}
\def\one{{\mathbbm1}}

\def\AA{\mathcal{A}}\def\BB{\mathcal{B}}\def\CC{\mathcal{C}}
\def\FF{\mathcal{F}}
\def\GG{\mathcal{G}}\def\HH{\mathcal{H}}
\def\KK{\mathcal{K}}

\def\PP{\mathcal{P}}\def\QQ{\mathcal{Q}}

\def\Gf{\mathfrak{G}}

\def\Nbb{\mathbb{N}}
\def\Rbb{\mathbb{R}}

\def\eps{\varepsilon}

\def\Pr {{\rm Pr}}

\def\bV {{\bf V}}

\def\bE {{\bf E}}

\def\bG {{\bf G}}

\newcommand{\dvar}{{d_{\mathrm{var}}}}

\begin{document}
\maketitle \addtolength{\baselineskip}{3pt}

\begin{abstract}
The colored neighborhood metric for sparse graphs was introduced by
Bollob\'as and Riordan \cite{BR}. The corresponding convergence
notion refines a convergence notion introduced by Benjamini and
Schramm \cite{BS}. We prove that even in this refined sense, the
limit of a convergent graph sequence (with uniformly bounded degree)
can be represented by a graphing. We study various topics related to
this convergence notion such as: Bernoulli graphings, factor of i.i.d.
processes and hyperfiniteness.
\end{abstract}

\newpage
\tableofcontents


\section{Introduction}
The theory of graph convergence is a recently emerging field. It
creates a link between combinatorics and analysis similarly as
F\"urstenberg's correspondence principle connects finite integer
sequences with measure preserving systems. Interestingly (or rather
unfortunately) there is no unified theory of graph convergence.
Instead there are various convergence notions that work well in
different situations. For example the theory of dense graph limits
\cite {LSz, LSz2, BCLSV1} works well if the number of edges is
quadratic in the number of vertices but it trivializes for sparser
graphs. On the other hand the Benjamini--Schramm limit \cite{BS} is
only defined for graphs which have a linear number of edges in terms
of the vertices. In the regime between linear and quadratic the
situation is more complicated.

In this paper we focus on the very sparse case were graphs have
degrees bounded by some fixed number $d$ (which we consider as fixed
throughout). According to Benjamini and Schramm, a graph sequence
$(G_n)_{n=1}^\infty$ is convergent if the distribution of the
isomorphism types of neighborhoods of radius $r$ (when a vertex is
chosen uniformly at random in $G_n$) converges for every fixed $r$.
This notion of convergence is called \emph{local convergence},
\emph{weak convergence} or  \emph{Benjamini--Schramm convergence}.

The following example illustrates why a different, stronger notion of
convergence is needed in some cases. For odd $n$, let $G_n$ be a
$d$-regular expander graph on $n$ nodes. For even $n$, let $G_n$ be
the disjoint union of two $d$-regular expander graphs on $n/2$ nodes.
Assume that the girth of $G_n$ tends to infinity. Then the sequence
$G_n$ is locally convergent, but clearly even and odd members of the
sequence are quite different, and it would be desirable to refine our
notion of convergence to distinguish them.

Bollob\'as and Riordan \cite{BR} introduced such a finer convergence
notion (i.e., fewer sequences are convergent). A graph sequence
$(G_n)_{n=1}^\infty$ is convergent in this sense if for every
$r,k\in\mathbb{N}$ and $\eps>0$ there is an index $l$ such that if
$n,m>l$, then for every coloring of the vertices of $G_n$ with $k$
colors, there is a coloring of the vertices of $G_m$ with $k$ colors
such that the distance between the distributions of colored
neighborhoods of radius $r$ in $G_n$ and $G_m$ is at most $\eps$.
This is equivalent to saying that $G_n$ and $G_m$ are close in the
colored neighborhood metric introduced in \cite{BR}. This finer
notion of convergence is sensitive to both local and global
properties of the graphs whereas the Benjamini--Schramm convergence
is only sensitive to local properties. For this reason we call this
notion \emph{local-global convergence}.

Benjamini and Schramm described a limit object for locally convergent
sequences in the form of an involution-invariant distribution on
rooted countable graphs with bounded degree. One can also describe
this limit object as a \emph{graphing} (Aldous and Lyons
\cite{AldousLyons}, Elek \cite{Elek}), which is a bounded degree
graph on a Borel probability space such that the edge set is Borel
measurable and it satisfies a certain measure preservation property.
(We will give a precise definition below.) Neighborhood statistics in
graphings can be defined by using the probability space structure on
the vertex set. Every involution-invariant distribution can be
represented by a graphing. We note that graphings are common
generalizations of bounded degree graphs and measure preserving
systems and so they are also interesting from an ergodic theoretic
point of view.

However, the graphing representing the limit object of a locally
convergent graph sequence is not unique: different graphings can
describe the same involution-invariant distribution. In other words,
a graphing contains more information than just the limiting
neighborhood distribution. This suggests that graphings can be used
to represent limit objects for more refined convergence notions.
Indeed, in the present paper we show that the limit of a local-global
convergent sequence can also be represented by a graphing in the
sense that the graphs in the sequence converge to the graphing in the
colored neighborhood metric. This means that for every local-global
convergent sequence we produce a graphing which contains both local
and global information about the graphs.

We highlight the importance of a special family of graphings called
\emph{Bernoulli graphings}. We show that with given local statistics,
the Bernoulli graphings contain the least global information. This
means that the global properties of a Bernoulli graphing can be
modeled with an arbitrary precision on any other graphing with the
same local statistics.
For a graph $G$, being close to a Bernoulli graphing in the
{local-global sense means that the local statistics of any coloring
on $G$ can be modeled by a randomized process called \emph{local
algorithm} or \emph{factor of i.i.d. process}.}

Roughly speaking, a hyperfinite graph sequence is a bounded degree
sequence whose members can be cut into small connected components
removing a small set of vertices (or equivalently edges). {We prove
that a locally convergent hyperfinite sequence is locally-globally
convergent, and its limit is a Bernoulli graphing. (This was proved
independently by Elek \cite{Elek2}). It is an interesting question
how to construct a non-hyperfinite sequence converging to a Bernoulli
graphing.}

\section{Local-Global convergence of bounded degree graphs}

A \emph{rooted graph} is a pair $(G,o)$ where $o$ is a vertex of a
graph $G$. The \emph{radius} of a rooted graph is the distance of the
farthest vertex in $G$ to $o$. We denote by $U^{r}$ the set of all
rooted graphs with radius at most $r$ (and all degrees bounded by
$d$). For an integer $r \ge 0$, and a vertex $v$ in a graph $G$, let
$N_{G,r}(v)$ denote the subgraph of $G$ rooted at $v$ and induced by
the vertices that are at a distance at most $r$ from $v$. Two rooted
graphs $(G,o)$ and $(G',o')$ are said to be \emph{isomorphic} if
there is an isomorphism from $G$ to $G'$ that maps $o$ to $o'$.

Given a finite graph $G$ and a radius $r\ge 0$, we can choose a node
$v \in V(G)$ uniformly and randomly, and consider the distribution of
$N_{G,r}(v)$. Let $P_{G,r}$ denote this probability measure on $U^r$.
We say that a sequence $(G_n)$ of finite graphs is \emph{locally
convergent} (or Benjamini--Schramm convergent) if $P_{G_n,r}$
converges to a limit distribution as $n\to\infty$, for every fixed
$r\ge 0$.

Denote the set of probability measures on a Borel space $X$ by
$M(X)$. Note that since $U^r$ is finite, all the usual distances on
$M(U^r)$ are topologically equivalent. We shall usually work with the
\emph{total variation distance} $\dvar$, defined (in general, for a
space $X$) by
$$
\dvar(\mu,\nu)=\sup_{A\subseteq X}|\mu(A)-\nu(A)|
$$
where $A$ runs through the Borel measurable sets.

To define our refinement of local convergence, we consider vertex
colorings. For a finite graph $G$, let $K(k,G)$ denote the set of all
vertex colorings with $k$ colors. Fix integers $k$ and $r$, and let
$U^{r,k}$ be the set of all triples $(H,o,c)$ where $(H,o)$ is a
rooted graph of radius at most $r$ and $c$ is an arbitrary
$k$-coloring of $V(H)$. Consider a finite graph $G$ together with a
$c \in K(k,G)$. Pick a random vertex $v$ from $G$. Then the
restriction of the $k$-coloring to $N_{G,r}(v)$ is an element in
$U^{r,k}$, and thus for the graph $G$, every $c\in K(k,G)$ introduces
a probability distribution on $U^{r,k}$ which we denote by
$P_{G,r}[c]$. Sometimes we  refer to the probability distributions
$P_{G,r}[c]$ (for $r \ge 0$) as  \emph{local statistics} of the
coloring $c$. Let
$$
Q_{G,r,k} :=\bigl\{P_{G,r}[c]:~ c \in K(k,G)\bigr\} \subseteq
M(U^{r,k}).
$$
These sets are similar to ``quotient sets'' introduced in
\cite{BCLSV2} for dense graphs, except that there only edges with the
given coloring were counted, while here we consider the colors on
larger neighborhoods. Notice that the sets $Q_{G,r,k}$ are finite,
and they are subsets of the finite dimensional space $\Rbb^{U^{r,k}}$
that is independent of the graph $G$.

\begin{definition}
\label{def:local-global} A sequence of finite graphs
$(G_n)_{n=1}^\infty$ with all degrees at most $d$ is called
\emph{locally-globally convergent} if for every $r,k\ge 1$, the
sequence $(Q_{G_n,r,k})_{n=1}^\infty$ converges in the Hausdorff
distance inside the compact metric space $(M(U^{r,k}),\dvar)$.
\end{definition}

{In other words, if $i$ and $j$ are large enough, then for every
$k$-coloring $c_i$ of $V(G_i)$ there is a $k$-coloring $c_j$ of
$V(G_j)$ so that the distributions of colored $r$-neighborhoods of
$(G_i,c_i)$ and $(G_j,c_j)$ are almost the same.}

Since compact subsets of a compact metric space form a compact space
with respect to the Hausdorff metric, it follows that every infinite
sequence of finite graphs contains a locally-globally convergent
subsequence.

\ignore{We can also define local-global convergence through
appropriate metrics. Let $G_1$ and $G_2$ be two finite graphs. First,
for a fixed depth $r\ge 1$ and number of colors $k\ge1$ we define
$d_\text{\rm sample}^{(r,k)}(G_1,G_2)$ as the infimum of all numbers
$a>0$ such that for every $k$-coloring $c_1$ of $V(G_1)$ there exists
a $k$-coloring $c_2$ of $V(G_2)$ such that $d_\text{\rm
var}\bigl(P_{G_1,r}[c_1], P_{G_2,r}[c_2]\bigr)\le a$, and vice versa.
We then define the \emph{nondeterministic sampling distance} as a
weighted average of these numbers:
\[
d_\text{\rm nd}(G_1,G_2) = \sum_{k=1}^\infty \sum_{r=1}^\infty
\frac{1}{2^{r+k}} d^{(r,k)}_\text{\rm sample}(G_1,G_2).
\]
It is easy to see that this defines a metric on finite graphs, and a
sequence is local-global convergent if and only if it is Cauchy in
this metric.}

Fixing $k=1$ in Definition~\ref{def:local-global}, we recover a
metric definition of Benjamini--Schramm convergence. It is easy to
construct examples of graph sequences  which are convergent in
Benjamini--Schramm  sense, but not locally-globally. However we do
not know whether $k=2$ would give a convergence notion equivalent to
local-global convergence.

It is natural to ask if we obtain a different convergence notion if
we replace vertex colorings by edge colorings or other locally
defined extra structures. It turns out that all local structures can
be encoded by vertex colorings, and thus they do not lead to
different convergence notions. As an example, we show how to encode
edge colorings by vertex colorings.

Let $G$ be a graph with all degrees at most $d$ and let
$c:~E(G)\rightarrow [k]$ be an edge coloring of $G$. It is easy to
see that there exists an edge coloring $c_1:~E(G)\rightarrow
[30d^3k]$ such that $c_1(e)\equiv c(e)$ modulo $k$ for every $e\in
E(G)$, and  if $c_1(e_1)=c_1(e_2)$ holds, then the edges $e_1$ and
$e_2$ are of distance at least $3$ in the edge graph of $G$. It is
clear that $c_1$ encodes the coloring $c$ in the sense that local
statistics of $c_1$ modulo $k$ give the local statistics of $c$. Let
$S$ denote the set of subsets of $[30d^3k]$ of size at most $d$. We
define the vertex coloring $c_2:~V(G)\rightarrow S$ by setting
$c_2(v)$ to be the set of $c_1$-colors of the edges incident to $v$.
Now it is easy to see that $c_2$ encodes the coloring $c_1$ in the
following way. If $e=(v,w)$ is an edge in $G$, then {$\{c_1(e)\}$} is
the intersection of the sets $c_2(v)$ and $c_2(w)$.

\section{Involution-invariant measures and graphings}\label{graphing}

Benjamini and Schramm \cite{BS} associated a limit object with every
locally convergent graph sequence as follows. Let $\Gf$ denote the
set of (isomorphism classes of) rooted, connected (possibly infinite)
graphs with all degrees at most $d$. For a rooted graph $(B,o)$ with
radius $r$, we denote by $\Gf(B,o)$ the set of all rooted graphs
$(G,o)$ such that $N_{G,r}(o)\cong (B,o)$. For a rooted graph
$(G,o)$, we define a neighborhood basis at $(G,o)$ as
$\Gf(N_{G,r}(o))$. These neighborhoods define a topology on $\Gf$. It
is easy to see that this is a compact separable space.

The Benjamini--Schramm limit of the locally convergent graph sequence
$(G_n)_{n=1}^\infty$ is a probability measure $\nu$ on the Borel sets
of $\Gf$, such that
\[
\lim_{n \to \infty}
P_{G_n,r}(B,o) = \nu(\Gf(B,o))
\]
for every $r\ge1$ and every rooted graph $(B,o)$ of radius  $r$.

Not every probability measure on $\Gf$ arises as the limit of a
convergent graph sequence. One property that all limits have is
called \emph{involution invariance} or \emph{unimodularity}. To
define this, let $\tilde{\Gf}$ denote the space of graphs in $\Gf$
with a distinguished edge incident to the root. Let
$\alpha:\tilde{\Gf}\rightarrow\tilde{\Gf}$ denote the continuous
transformation that moves the root to the other endpoint of the
distinguished edge. For every probability measure $\mu$ on $\Gf$,
define $\mu^*$ to be the unique probability measure on $\Gf$ such
that $d\mu^*/d\mu(G)$ is proportional to the degree of the root in
$G$. Define the probability measure $\tilde{\mu}$ on $\tilde{\Gf}$ by
first picking a $\mu^*$-random graph, and then distinguishing a
random edge incident to the root. The measure $\mu$ is called
\emph{involution-invariant} if $\tilde{\mu}$ is invariant under
$\alpha$. Involution-invariant measures on $\Gf$ form a closed set in
the weak topology.

Let $G$ be a finite graph, and let the probability measure $\nu$ on
the Borel sets of $\Gf$ be defined as $\nu(\Gf(B,o))=P_{G,r}(B,o)$
for every $r\ge1$ and every rooted graph $(B,o)$ of radius $r$. It is
easy to see that $\nu$ is involution-invariant. It follows that every
measure on $\Gf$ that is the limit of finite graphs is
involution-invariant. Aldous and Lyons~\cite{AldousLyons} conjectured
that all involution-invariant measures arise as graph limits. The
Aldous-Lyons conjecture is considered to be one of the most important
open problems in this area.

In the dense setting, the set of the symmetric measurable maps
$w:[0,1]^2 \rightarrow [0,1]$ were used to generalize the concept of
graphs and describe graph limits \cite{LSz}. For local-global
convergence {(Definition \ref{def:local-global})}, graphings serve
this purpose.

\begin{definition}\label{defgraphing}
Let $X$ be a Polish topological space and let $\nu$ be a probability
measure on the Borel sets in $X$. A {\em graphing} is a graph $\GG$
on $V(\GG)=X$ with Borel measurable edge set $E(\GG)\subset X\times
X$ in which all degrees are at most $d$ and
\begin{equation}
\label{eq:degreeCond} \int_A e(x,B) d\nu(x) = \int_B e(x,A) d\nu(x)
\end{equation}
for all measurable sets $A,B\subseteq X$, where $e(x,S)$ is the
number of edges from $x\in X$ to $S\subseteq X$.
\end{definition}

Note that every finite graph $G$ is a graphing where $X=V(G)$ and
$\nu_\GG$ is the uniform distribution on $V(G)$.

If  \eqref{eq:degreeCond} holds, then $\eta^*(A\times B)=\int_A
e(x,B) d\nu(x)$ defines a measure on the Borel sets of $X\times X$.
This measure is concentrated on $E(\GG)$, symmetric in the two
coordinates, and its marginal $\nu^*$ satisfies $(d\nu^*/d\nu)(x)
=\deg(x)$. Normalizing by $d_0=\int_X \deg(x)\,dx$, we get a
probability distribution $\eta$ on the set of edges. We can generate
a random edge from $\eta$ by selecting a random point $v$ from
$\nu^*$ and selecting uniformly a random edge incident with $v$.
Conversely, if $\GG$ is a Borel graph and we have a measure $\eta^*$
on $X\times X$ that is concentrated on $E(\GG)$, so that
$\eta^*(A\times B)=\int_A e(x,B) d\nu(x)$, then \eqref{eq:degreeCond}
follows by Fubini's theorem, and so $\GG$ is a graphing.

Let $\GG$ be a graphing (of degree at most $d$) on the probability
space $(X,\nu)$. Then it induces a measure $\mu_{\GG}$ on $\Gf$: pick
a random element $x\in X$ and take its connected component $\GG_x$
rooted at $x$. It is easy to see that $\mu_{\GG}$ is an
involution-invariant measure. (In fact, \eqref{eq:degreeCond} just
expresses this property.)

Let $\GG$ be a graphing as in Definition \ref{defgraphing}. A {\em
vertex coloring} of $\GG$ with $k$ colors is a measurable function
$c:~X\rightarrow [k]$. The set of all such colorings  will be denoted
by $K(k,\GG)$. We define $P_{\GG,r}[c]$ and $Q_{\GG,r,k}$ in a
similar way as in a finite graph. (Notice that it makes sense to talk
about a random vertex in $\GG$.) The set $Q_{\GG,r,k}$ is a subset of
the finite dimensional space $\Rbb^{U^{r,k}}$, but in general it
might be infinite and not necessarily closed; we will often use its
closure $\overline{Q}_{\GG,r,k}$ (see
Question~\ref{ques:closedness}).

Now we are ready to state our main theorem.

\begin{theorem}
\label{thm:graphing} Let $(G_i)_{i=1}^\infty$ be a local-global
convergent sequence of finite graphs with all degrees at most $d$.
Then there exists a graphing $\GG$ such that $Q_{G_n,r,k}\to
Q_{\GG,r,k}$ $(n\to\infty)$ in Hausdorff distance for every $r$ and
$k$.
\end{theorem}

To what degree is the limit object determined? This question leads to
different notions of ``isomorphism'' between graphings.

\begin{definition}
Let  $(\GG_1,X_1,\nu_1)$ and $(\GG_2,X_2,\nu_2)$ be graphings.
\begin{itemize}
\item  They are called \emph{locally equivalent} if for every
$r\in\mathbb{N}$, the distribution of $N_{\GG_1,r}(x_1)$ is the same
as the distribution of $N_{\GG_2,r}(x_2)$ for random $x_1\in X_1$ and
$x_2\in X_2$.

\item  They are called \emph{locally-globally equivalent}
if $\overline{Q}_{\GG_1,r,k}=\overline{Q}_{\GG_2,r,k}$ for every
$r,k\in\mathbb{N}$.
\end{itemize}
\end{definition}

Local equivalence of two graphings means that they induce the same
involution-invariant measure on $\Gf$. Local-global equivalence
implies local equivalence by setting $k=1$.

\begin{definition}[Local-global partial order]
Assume that $\GG_1$ and $\GG_2$ are two graphings of maximal degree
at most $d$. We say that $\GG_1\prec\GG_2$ if
$\overline{Q}_{\GG_1,r,k}\subseteq \overline{Q}_{\GG_2,r,k}$ for
every $r,k\ge1$. In particular, $\GG_1$ and $\GG_2$ are
locally-globally equivalent if and only if both $\GG_1\prec\GG_2$ and
$\GG_2\prec\GG_1$ hold.
\end{definition}

{In the setting of group actions, this partial order means the same
as ``weak containment'' of the corresponding group actions, and
local-global equivalence corresponds to ``weak equivalence'' (Kechris
\cite{Kech}).}

Recall that a measurable map $\phi:(X,\mu) \to (Y, \nu)$ is called
{\it measure-preserving} if $\mu (\phi^{-1}(A))=\nu(A)$ for every
measurable set $A \subseteq Y$. An easy way to prove a relation
$\GG_1\prec\GG_2$ between two graphings is the following. We call a
measure preserving map $\phi:~V(\GG_1)\to V(\GG_2)$ a \emph{local
isomorphism} if restricted to any connected component of $\GG_1$, we
get an isomorphism with a connected {component} of $\GG_2$. Clearly
local isomorphisms can be combined. However, a local isomorphism may
not be invertible! It is easy to see that the existence of a local
isomorphism $\GG_1\to\GG_2$ implies that $\GG_1$ and $\GG_2$ are
locally equivalent, and $\GG_2\prec\GG_1$.

\begin{example}
Let $G$ be a finite connected graph, and $G \cup G$ denote the
disjoint union of $G$ with itself. The function $\phi:~V(G \cup G)\to
V(G)$ that maps both copies of $G$ in $G \cup G$ isomorphically to
$G$ is a (non-invertible) local isomorphism. Consequently $G \cup G$
and $G$ are locally equivalent, and $G \prec G \cup G$. {However,}
$G$ and $G \cup G$ are not locally-globally equivalent.
\end{example}

We shall study the local-global equivalence and the local-global
partial order in Sections \ref{SEC:BERN} and \ref{SEC:JOINS}. In
particular, we will show that among all graphings in a local
equivalence class, there is always a smallest one {and a largest one}
in this partial order.

\section{Local limits of decorated graphs}\label{benschramm}

In this section we extend the formalism behind the Benjamini--Schramm
limits for the case when vertices are decorated by elements from a
compact space. Let $C$ be a second countable compact Hausdorff space.
Let $\Gf(C)$ denote the space of (isomorphism classes of) rooted,
connected (countable) graphs with all degrees at most $d$ such that
the vertices are decorated by elements from $C$. So the points of
$\Gf(C)$ are triples $(G,o,c)$, where $G$ is a connected countable
graph, $o\in V(G)$, and $c:~V(G)\to C$. If $C$ is the trivial (one
point) compact space, then $\Gf(C)$ can be identified with the space
$\Gf$ defined earlier. Two important special cases for us will be
when $C=[0,1]$ (assigning $[0,1]$-weights to vertices), and $C=[k]$
(coloring vertices by $k$ colors). With a slight abuse of notation,
these will be denoted by $\Gf[0,1]$ and $\Gf[k]$.

We put a compact topology on $\Gf(C)$ by specifying a basis of it.
Let $r$ be an arbitrary natural number and $(H,o)$ be a finite rooted
graph of radius $r$. Assume furthermore that every vertex $v$ of
$(H,o)$ is decorated by an open set $U_v$ in $C$. Let $S$ be the
collection of all $(G,o,c)\in\Gf(C)$ where the neighborhood
$N_{G,r}(o)$ is isomorphic to $(H,o)$, and furthermore there is an
isomorphism $\alpha:~N_{G,r}(o)\rightarrow (H,o)$ such that $c(v) \in
U_{\alpha(v)}$ for every $v\in N_{G,r}(o)$. It is easy to see that
$\Gf(C)$ with this topology is a compact, second countable, Hausdorff
space. As a consequence, probability measures on $\Gf(C)$ form a
compact space in the weak topology.

Let $G$ be a finite graph with all degrees at most $d$ in which the
vertices are $C$-labeled. We can construct a probability measure
$\mu_{G}$ on $\Gf(C)$ by putting a root $o$ on a randomly chosen
vertex $v\in V(G)$ and keeping only the connected component of the
root. A sequence $(G_n)_{n=1}^\infty$ of $C$-labeled graphs is called
\emph{locally convergent} if the corresponding measures
$\{\mu_{G_n}\}_{n=1}^\infty$ converge in the weak topology to some
measure $\mu$. The measure $\mu$ is the limit object of the sequence.

We define involution-invariance completely analogously to the
undecorated case, simply replacing $\Gf$ by $\Gf(C)$ everywhere.
Involution-invariant measures on $\Gf(C)$ form a closed set in the
weak topology. It follows that if $\mu$ is a measure on $\Gf(C)$ that
is the limit of finite $C$-decorated graphs, then it is
involution-invariant.

A \emph{$C$-decorated graphing} is a graphing  $\GG$ together with a
Borel function $c:~V(\GG)\rightarrow C$. Similarly as in the
undecorated case, every $C$-decorated graphing defines an
involution-invariant distribution. The measure $\mu_{\GG,c}$ on
$\Gf(C)$ is created by picking a random element $x\in V(\GG)$, and
taking its connected component $\GG_x$ rooted at $x$ together with
the vertex labels given by the restriction of $c$ to $V(\GG_x)$. It
is easy to see that $\mu_{\GG,c}$ is an involution-invariant measure.

\begin{remark}
\label{rem:UniversalGraphing} We can define a Borel graph on
$\Gf(C)$. The edge set $\mathcal{E}(C)$ of this graph consists of
pairs $((G,o_1,c),(G,o_2,c))\in \Gf(C) \times \Gf(C)$ such that
$(o_1,o_2)$ is an edge in $G$. Note that loop edges can arise in this
graph. For example if there is an automorphism of $(G,c)$ which takes
$o_1$ to its neighbor $o_2$, then $(G,o_1,c)$ is identified with
$(G,o_2,c)$ in $\Gf(C)$. In general it is not true that every
involution-invariant measure $\nu$ on $\Gf(C)$ turns this graph into
a graphing. This is due to the problem with automorphisms which also
lead to loops. However it is not hard to show that if for an
involution-invariant measure $\nu$, with probability one, a
$\nu$-random connected component has no automorphisms, then we get a
graphing $(\Gf(C),\nu,\mathcal{E}(C))$. One important role of
appropriate decorations is to break symmetries, and make this graph a
graphing.
\end{remark}

\section{A regularization lemma}

The following lemma is the main ingredient in proving
Theorem~\ref{thm:graphing}. It serves as a ``regularity lemma'' in
our framework for bounded degree graphs.

\begin{lemma}[Regularization]
\label{lem:keyGraphing} For positive integers $r,k$ and real number
$\eps>0$, there exists an integer $t_{r,k,\eps}$ such that the
following holds. For every graph $G$ with all degrees at most $d$,
there exists a $t_{r,k,\eps}$-vertex coloring $q$ of $G$ which
satisfies the following conditions.
\begin{itemize}

\item If $q(v)=q(w)$, then either $v=w$ or the distance of $v$
and $w$ in $G$ is at least $r+1$;

\item For every $g \in K(k,G)$, there exists
    $\alpha:[t_{r,k,\eps}]\rightarrow [k]$ such that
$$\dvar(P_{G,r}[g] , P_{G,r}[\alpha\circ q]) \le\eps.$$

\end{itemize}
\end{lemma}
\begin{proof}
The space $M(U^{r,k})$ is a bounded dimensional compact set with the
topology generated by $\dvar$. Let $N$ be an $\eps/2$-net in
$M(U^{r,k})$ in $\dvar$. Let $N_G$ be the subset of points in $N$
that are at most $\eps/2$ far from a point of the form $P_{G,r}[g]$
for some $g\in K(k,G)$. For each $a \in N_G$, we choose a
representative $x_a=P_{G,r}[g_a]$ such that $\dvar(a,x_a)\leq
\eps/2$. It is clear that for every $g\in K(k,G)$, there is a point
$x_a$ such that $d_{\text{\rm var}}(P_{G,r}[g],x_a)\leq\eps$. Let $f$
be the common refinement of all the partitions $\{g_a\}_{a\in N_G}$.
Clearly $f$ has a bounded number of partition sets in terms of $r,k,
\eps$ and $d$ and it satisfies the second condition.

Now we further refine $f$ to satisfy the first condition. Let $f'$ be
a proper coloring of the graph $G$ with $(d+1)^r$ colors in which
every two vertices in distance at most $r$ receive different colors.
The common refinement $q$ of $f$ and $f'$ satisfies both conditions.
\end{proof}

\section{Proof of the main theorem}

Now we introduce the space $X$ which will serve as a universal Borel
space for the limit graphings of sequences of finite graphs with all
degrees at most $d$. Consider the compact space
$C=\prod_{k,r,n}[t_{r,k,1/n}]$ with the product topology where
$t_{r,k,1/n}$ are defined according to Lemma~\ref{lem:keyGraphing}.
We denote by $X$ the compact space $\Gf(C)$ and by $E\subset X\times
X$ the set of edges $((G,o_1,c),(G,o_2,c))$ such that $(o_1,o_2)$ is
an edge in $G$ (See Remark~\ref{rem:UniversalGraphing}). Let
$q:X\rightarrow C$ be the function defined as $q:(G,o,c) \mapsto
c(o)$. Furthermore for $r,k,n\in\mathbb{N}$,  define the coloring
$q_{r,k,n}:X\rightarrow [t_{r,k,1/n}]$ as the composition of $q$ with
the projection  to the coordinate $(r,k,n)$ in $C$.

Let $(G_i)_{i=1}^\infty$ be a local-global  convergent sequence of
graphs with all degrees at most $d$. For each $G_i$ and triple
$(r,k,n)\in\mathbb{N}^3$, we choose a coloring
$q^i_{r,k,n}:V(G_i)\rightarrow [t_{r,k,1/n}]$ guaranteed by
Lemma~\ref{lem:keyGraphing}. Let $q_i:~V(G_i)\rightarrow C$ be
defined as $\prod_{r,k,n}\{q^i_{r,k,n}(v)\}\in C$. As described in
Section~\ref{benschramm}, each graph $G_i$ together with the coloring
$q_i$ defines a probability measure $\mu_i$ on $X$ by putting the
root on a random vertex of $G_i$ and keeping only the connected
component of the root.

By choosing a subsequence from $(G_i)_{i=1}^\infty$ we can assume
that the sequence $\{\mu_i\}_{i=1}^\infty$ weakly converges to a
probability distribution $\mu$ on $X$. Our goal is to show that the
Borel graph $(X,E)$ with the measure $\mu$ is a graphing which
represents the local-global limit of $(G_i)_{i=1}^\infty$.

Let us first observe that for a $\mu$-random element $(G,o,c)$ in
$(X,\mu)$, with probability one, the vertex labels $\{c(v) : v \in
V(G)\}$ are all different. This follows from the fact that the
colorings $q^i_{r,k,n}$ separate points in $G_i$ that are closer than
$r+1$, and that this property is preserved in the limit. This means
that if $v,w\in V(G)$ are of distance $r$, then with probability one
their colors projected to the coordinate $(r,k,n)$ (where $k,n$ are
arbitrary) are different.

\begin{lemma}
The measurable graph $(X,E,\mu)$ is a graphing.
\end{lemma}
\begin{proof}
Let us introduce the measures $\{\eta_i^*\}_{i=1}^\infty$, similarly
as in Section~\ref{graphing}, by
\[
\eta_i^*(A\times B)=\int_A e(x,B) d\mu_i(x),
\]
where $A,B\subseteq X$ are measurable, and $e(x,B)$ is the number of
edges $(x,y)\in E$ with $y\in B$. We define $\eta^*$ analogously as
$\eta^*(A\times B)=\int_A e(x,B) d\mu(x)$.

Assume that $A,B\subset X$ are open-closed sets. The weak convergence
of $\{\mu_i\}_{i=1}^\infty$ implies that
$\lim_{i\to\infty}\eta_i^*(A\times B)=\eta^*(A\times B)$ and
$\lim_{i\to\infty}\eta_i^*(B\times A)=\eta^*(B\times A)$. Note that
$\eta^*_i(A\times B)=\eta^*_i(B\times A)$, since both are equal (up
to normalization by $|V(G_i)|$) to the number of edges between the
sets $\{v|(G_i,v,q_i)\in A\}$ and $\{v|(G_i,v,q_i)\in B\}$. Here we
used the fact that the vertex labels $q_i(\cdot)$ are all different
and thus automorphisms of $G_i$ cannot cause any problems. We obtain
that $\eta^*(B\times A)=\eta^*(A\times B)$, and since such product
sets generate the whole $\sigma$-algebra on $X\times X$, the proof is
complete.
\end{proof}

\begin{lemma}\label{prlem1}
The probability distributions $P_{G_i,r}[q^i_{r,k,n}]$ converge to
$P_{\GG,r}[q_{r,k,n}]$ {as $i\to\infty$} for every fixed triple
$r,k,n\in\mathbb{N}$.
\end{lemma}
\begin{proof}
Pick a $\mu$-random point $x=(G,o,c) \in X$. Let the rooted graph
$\GG_x$ be the connected component of $x$ in the graphing $\GG$
rooted at $x$. There is a natural vertex coloring on $\GG_x$ which is
the restriction of the function $q$ to the vertices of $\GG_x$. So
$\GG_x$ can be regarded as an element in $X$. We claim that with
probability one $x=(G,o,c)$ is isomorphic (in a root and label
preserving way) to $(\GG_x,q|_{\GG_x})$. Indeed with probability one
all the vertex labels of $G$ are different, and in this case the map
given by $v\mapsto(G,v,c)$ defines a decoration-preserving
isomorphism between $(G,o,c)$ and $\GG_x$. (The fact that the vertex
labels in $G$ are all different guarantees that the map is one to
one.)

We conclude that the probability distribution
$P_{\GG,r}[q_{r,k,n}]$ is the same as the distribution of
$(N_{G,r}(o),c_{r,k,n})$ where $(G,o,c)$ is a $\mu$-random element in $X$, and
$c_{r,k,n}$ is the
projection of $c$ to the coordinate $(r,k,n)$. The lemma now follows from the
weak
convergence of $\{\mu_i\}_{i=1}^\infty$ to $\mu$.
\end{proof}

\begin{lemma}
For every $r,k\in\mathbb{N}$ {and $\eps>0$} there is an index $i_0$
such that for every $i\geq i_0$ and $c\in K(k,G_i)$, there is a
$k$-coloring $c'$ of $X$ such that $\dvar(P_{G_i,r}[c],
P_{\GG,r}[c']) \le\eps.$
\end{lemma}

\begin{proof}
Let $n\ge 2/\eps$. By Lemma~\ref{prlem1} there is an index $i_0$ such
that
\begin{equation}\label{claim1eq}
\dvar(P_{G_i,r}[q^i_{r,k,n}],
P_{\GG,r}[q_{r,k,n}])\leq \frac{\eps}{2}
\end{equation}
for every index $i\geq i_0$. Let $i\geq i_0$ be arbitrary, and let
$c\in K(k,G_i)$ be a $k$-coloring of $G_i$. Then by
Lemma~\ref{lem:keyGraphing} there is a map
$\alpha:~[t_{r,k,\eps/2}]\rightarrow[k]$ such that
\[
\dvar(P_{G_i,r}[c],P_{G_i,r}[\alpha\circ q^i_{r,k,n}])\leq\frac1n \le
\frac{\eps}{2}.
\]
The definition of the total variation distance and (\ref{claim1eq})
imply that
\[
\dvar(P_{G_i,r}[\alpha\circ q^i_{r,k,n}],
P_{\GG,r}[\alpha\circ q_{r,k,n}])\leq \frac{\eps}{2}.
\]
Hence $c'=\alpha\circ q_{r,k,n}$ satisfies the required condition.
\end{proof}

\begin{lemma}
For every coloring $c\in K(k,\GG)$, $r\in\mathbb{N}$ and {$\eps>0$}
there is an index $i_0$ such that for every $i\geq i_0$, there is a
coloring $c'\in K(k,G_i)$ with $\dvar( P_{G_i,r}[c'], P_{\GG,r}[c])
\le\eps.$
\end{lemma}

\begin{proof}
Let $c:X\rightarrow [k]$ be a Borel coloring. Then for every
$\delta>0$, there is a continuous coloring $c_\delta:~X\rightarrow
[k]$ such that $|\mu(c^{-1}(a) \triangle
c_\delta^{-1}(a))|\leq\delta$  for all $1\leq a\leq k$. Taking
$\delta$ to be sufficiently small, we have
\begin{equation} \label{eq:firstApprox}
\dvar( P_{\GG,r}[c_\delta],P_{\GG,r}[c]) \le \frac{\eps}{2}.
\end{equation}

Let the graphing $\GG_i$ be the same as the graphing $\GG$ with the
only difference that the measure $\mu$ is replaced by $\mu_i$. Since
$\{\mu_i\}_{i=1}^\infty$ converges weakly to $\mu$ and $c_\delta$ is
continuous, there is an index $i_0$ such that if $i \geq i_0$, then
\begin{equation} \label{eq:secondApprox}
\dvar(P_{\GG_i,r}[c_\delta], P_{\GG,r}[c_\delta] ) \le \frac{\eps}{2}.
\end{equation}
The coloring $c_\delta$ induces a coloring $f^i_\delta$ on $G_i$
which assigns to every vertex $v\in V(G_i)$ the $c_\delta$ color of
the rooted graph $(G_i,v,q_i)\in X$. Then we have
$P_{G_i,r}[f^i_\delta] \equiv P_{\GG_i,r}[c_\delta]$. {Together with
(\ref{eq:firstApprox}) and (\ref{eq:secondApprox}), this completes
the proof.}
\end{proof}

\section{Bernoulli graphings and Bernoulli graph
sequences}\label{SEC:BERN}

Probably the most fundamental graphing construction is the Bernoulli
graphing corresponding to an involution-invariant measure. These
graphings are closely related to factor of i.i.d. processes and local
algorithms. In this chapter we explain their role in local-global
convergence.

\begin{definition}[Bernoulli graphings]
Let $\mu$ be an involution-invariant measure on $\Gf$. Let $\nu$ be
the probability measure on $\Gf[0,1]$ produced by putting independent
random {weights} from $[0,1]$ on the nodes of a $\mu$-random graph.
(Note that different choices of the weights can lead to the same
point of $\Gf[0,1]$, if they can be transformed into each other by an
automorphism of the $\mu$-random rooted graph.) The triple
$(\Gf[0,1],\nu,\mathcal{E}[0,1])$ as defined in
Remark~\ref{rem:UniversalGraphing} will be called the \emph{Bernoulli
graphing} corresponding to $\mu$, and denoted by $\BB_\mu$.
\end{definition}

It is not hard to see that $\BB_\mu$ is a graphing and it represents
the involution-invariant distribution $\mu$ (Elek \cite{Elek}).

\begin{remark}\label{BERN}
Perhaps it would be more natural to decorate the nodes of the
$\mu$-random graph by independent bits, or more generally, by colors
from $[k]$ for some fixed $k\ge2$. This would yield an
involution-invariant distribution on $\Gf[k]$, but the graph
$(\Gf[k],\mathcal{E}[k])$ together with this distribution would not necessarily
form a graphing.
\end{remark}

We define the Bernoulli graphing $\BB_\GG$ corresponding to an
arbitrary graphing $\GG$ as the Bernoulli graphing defined by the
involution-invariant distribution induced by $\GG$ on $\Gf$. Clearly $\GG$ and
$\BB_\GG$ are locally equivalent.

\begin{example}
A simple example for a Bernoulli graphing is provided by the
involution-invariant measure which is concentrated on a single
$d$-regular rooted tree. Let $T$ denote the rooted $d$-regular tree,
and let $(X,\nu)$ be the probability space in which we put
independent random {weights} from $[0,1]$ on the vertices of $T$. Two
points of $X$ are connected in $\GG$ if they can be obtained from
each other by moving the root to a neighboring vertex. It seems to be
an interesting problem to decide whether the sets $Q_{\GG,r,k}$ are
all closed (see also Question~\ref{ques:closedness}).
\end{example}

The following is a related construction. For every graphing $\GG$ on
the probability space $(X,\nu)$, we define its \emph{Bernoulli lift}
$\GG^+$ as follows. The underlying set $X^+$ of $\GG^+$ will be pairs
$(x,\xi)$, where $x\in X$ and $\xi:~V(\GG_x)\to[0,1]$ assigns weights
from $[0,1]$ to the vertices of the connected component $\GG_x$
rooted at $x$. We connect $(x,\xi)$ to $(y,\upsilon)$ if $y$ is a
neighbor of $x$ and $\xi=\upsilon$. (Note that if $y$ is a neighbor
of $x$, then $\GG_x=\GG_y$.) The measure on $X^+$ is defined as
follows. To generate a random element of $X^+$, one picks a
$\nu$-random point $x\in X$, and then assigns independent random
weights $\xi(u)$ to the nodes $u$ of $\GG_x$.

We define two maps $\phi:~V(\GG^+)\to V(\GG)$ and $\psi:~V(\GG^+)\to
V(\BB_\GG)$ by $\phi(x,\xi)=x$ and $\psi(x,\xi)=
\bigl(\GG_x,\xi\bigr)$. It is easy to check that the maps $\phi$ and
$\psi$ are local isomorphisms. This implies that graphing $\GG$ is
locally equivalent to its Bernoulli lift $\GG^+$ as well as its
Bernoulli graphing $\BB_\GG$.

Our main goal in this section is to describe the relationship between
$\GG$, $\BB_\GG$ and $\GG^+$ from the point of view of local-global
equivalence.

\begin{definition}
A graphing is called \emph{atom-free} if its underlying probability space
contains no mass points.
\end{definition}

\begin{remark}
\label{rem:continuity} Note that no finite graph corresponds to an
atom-free graphing. Using the graphing property
\eqref{eq:degreeCond}, it is easy to see that if a graphing contains
an atom, then this belongs to a finite component. If $\GG$ is the
local limit of a sequence of connected graphs $(G_n)_{n=1}^\infty$
with $V(G_n)| \to \infty$, then all its components are infinite, and
hence it is atom-free. On the other hand, if the union of finite
components of a graphing has positive weight, then merging isomorphic
finite components we get atoms. Furthermore, if $\GG$ is the
local-global limit of graphs $(G_n)_{n=1}^\infty$ (not necessarily
connected) with $|V(G_n)| \to \infty$, then $\GG$ is atom-free. This
follows from the observation that a graphing is atom-free if and only
if its points have a Borel $k$-coloring with equal color classes for
every $k$.
\end{remark}

The following is our main result in this section.
\begin{theorem}\label{THM:BERLIFT}
Every atom-free graphing is local-global equivalent to its Bernoulli lift.
\end{theorem}

The map $\psi:~V(\GG^+)\to V(\BB_\GG)$ defined above is a local
isomorphism from $\GG^+$ to $\BB_\GG$. Thus we have the relation
$\BB_\GG\prec\GG^+$, which implies by Theorem \ref{THM:BERLIFT}:

\begin{corollary}[Minimality of Bernoulli graphings]\label{cor:bermin}
For every atom-free graphing $\GG$, we have $\BB_\GG\prec\GG$.
\end{corollary}

{In other words, Bernoulli graphings are minimal elements in the set
of atom-free graphings in their local equivalence class. A group
theoretical analogue of this fact was obtained by Ab\'ert and Weiss
in \cite{AW}.}

In an algorithmic setting, a Borel coloring of $\GG^+$ can be
considered as a coloring that depends not only on the graph, but also
on a random real number at each point. To be able to imitate this in
$\GG$, we have to construct ``random-like'' colorings on $\GG$. For
technical reasons, we have to deal with graphings that already have a
Borel coloring.

\begin{definition}[Quasirandom colorings]
Let $\GG$ be a graphing on the space $(X,\nu)$, and let $h:~X\to[l]$
be a Borel coloring. Let $\mu_{r,h,k}$ be the probability
distribution on $U^{r,kl}$ obtained from $\nu$ by considering the
$r$-neighborhood of a random element $x\in X$ and decorating its
vertices by random independent elements from $[k]$ (in addition to
the given $l$-coloring $h$). We say that a measurable coloring
$c:~X\rightarrow [k]$ is \emph{$(r,\eps)$-quasirandom} if $d_{\rm
var}(P_{\GG,r}[c\times h],\mu_{r,h,k})\leq\eps$ where  $c\times h$
denotes the $kl$-coloring with pairs of colors $(c(x),h(x))$.
\end{definition}

\begin{lemma}[Existence of quasirandom colorings]\label{quasirand}
Let $\GG$ be a atom-free graphing on the space $(X,\nu)$. Then for
every $k,r,l\in\mathbb{N}$, $\eps>0$ and Borel $l$-coloring $h$,
there is an $(r,\eps)$-quasirandom coloring $c:~X\rightarrow [k]$ of
$(\GG,h)$.
\end{lemma}
\begin{proof}
Let $C=\{0,1\}^\mathbb{N}$ be the Cantor set with the uniform
measure. Since $(X,\nu)$ has no mass points, there is a measurable
equivalence between $C$ and $X$, without loss of generality, we can
identify the two spaces, assume that $X=C$. Let
$\pi_i:C\rightarrow\{0,1\}^i$ be the projection onto the first $i$
coordinates. The map $\pi_i$ is measure preserving if we consider the
uniform measure on $\{0,1\}^i$. Fix $k,r\in\mathbb{N}$, and let
$g_i:\{0,1\}^i\rightarrow [k]$ be a uniform random coloring of
$\{0,1\}^i$ with $k$ colors. Our goal is to show that if $i$ is
sufficiently large, then with a large probability $g_i\circ\pi_i$ is
$(r,\eps)$-quasirandom.

\begin{claim}\label{CLAIM:QRAND}
For every $\eps_1>0$ and $n\in\mathbb{N}$, there is an index $j$ such
that if $x_1,\dots,x_n\in X$ are independent $\nu$-random points,
then with probability $1-\eps_1$, the map $\pi_j$ separates all the
points in $\cup_{i=1}^n N_{\GG,r}(x_i)$.
\end{claim}

It is easy to see that $\pi=(\pi_1,\pi_2,\dots)$ separates the points
of $\cup_{i=1}^nN_{\GG,r}(x_i)$ with probability {$1$} on $X^n$. Let
$Y_j$ denote the set of points $(x_1,x_2,\dots,x_n)$ in $X^n$ for
which $\pi_j$ separates the points in $\cup_{i=1}^n N_{\GG,r}(x_i)$.
{Then} $Y_j$ is an increasing chain of measurable sets such that
$\nu(\cup_{i=1}^\infty Y_i)=1$. This shows that for some index $j$,
we have $\nu(Y_j)>1-\eps_1$ and {completes} the proof of
Claim~\ref{CLAIM:QRAND}.

\smallskip

Let $x=(x_1,\dots,x_n)\in X^n$ and let $g$ be a $k$-coloring
$\cup_{i=1}^n N_{\GG,r}(x_i)$. Let us say that $x$ is \emph{representative} if
the
distribution of the $l$-colored neighborhood $N_{\GG,h,r}(x_t)$ for a
random $t\in[n]$ is $\eps/6$-close to the distribution $\mu_{r,h} :=
P_{\GG,r}[h]$.
Let us say that $(x,g)$ is \emph{representative} if the distribution
of the $kl$-colored neighborhood $(N_{\GG,h,r}(x_t),g)$ is
$\eps/3$-close to the distribution $\mu_{r,h,k}$.

Let $x=(x_1,\dots,x_n)\in X^n$ be chosen randomly and independently
from the distribution $\nu$. We note that with probability $1$, the
neighborhoods $N_{\GG,r}(x_i)$ are disjoint. If $n$ is large enough,
then (just by the Law of Large Numbers)
\[
\Pr_x(x~\text{representative})\ge 1-\frac\eps6.
\]
Hence if $g$ is a uniform random $k$-coloring of $\cup_{i=1}^n
N_{\GG,r}(x_i)$, and $n$ is large enough, then (by the Law of Large
Numbers again), we have
\[
\Pr_{x,g}((x,g)~\text{representative})\ge 1-\frac\eps3.
\]
Let us fix $n$ so that this holds.

Next, using Claim~\ref{CLAIM:QRAND}, we fix $j$ so that (for a random $x$)
$\pi_j$
separates all the points in $\cup_{i=1}^n N_{\GG,r}(x_i)$ with
probability at least $1-\eps/3$. Whenever this happens, the
restriction of $g_j\circ \pi_j$ to $\cup_{i=1}^n N_{\GG,r}(x_i)$ is a
uniform random $k$-coloring. In  other words, we can generate a
uniform random $k$-coloring of $\cup_{i=1}^n N_{\GG,r}(x_i)$ by
restricting $g_j\circ \pi_j$ to it if $\pi_j$ separates it, and
randomly $k$-coloring it otherwise. Thus
\[
\Pr_{x,g_j}((x,g_j\circ \pi_j)~\text{representative}) \ge
\Pr_{x,g}((x,g)~\text{representative})-\frac\eps3\ge 1-\frac{2\eps}3.
\]
It follows that there is at least one $k$-coloring $g_j$ for which
\[
\Pr_x((x,g_j\circ \pi_j)~\text{representative}) \ge 1-\frac{2\eps}3.
\]
Let us fix such a $g_j$. Then $c=g_j\circ \pi_j$ is an
$(r,\eps)$-quasirandom $k$-coloring of $X$. In fact, we can generate
a random point of $x$ by first generating $n$ independent random
points $x_1,\dots,x_n$ and choosing one of them, $x_t$, uniformly at
random. Then with probability at least $1-2\eps/3$, $(x,g_j \circ \pi_j
)$ is representative, and whenever this happens,  the
distribution of the $kl$-colored neighborhood
$(N_{\GG,h,r}(x_t),g_j\circ \pi_j)$ is $\eps/3$-close to the
distribution $\mu_{r,h,k}$. It follows that the total variation
distance of $(N_{\GG,h,r}(x_t),g_j\circ \pi_j)$ from $\mu_{r,h,k}$,
when $x_t$ is also randomly chosen, is at most $\eps$.
\end{proof}

Our next lemma shows that we can approximate any measurable
$k$-coloring of $\GG^+$ by a $k$-coloring that is {``locally
computable'' in the sense that the color of a node depends only on a
colored neighborhood of the node, and it depends only on a discrete
approximation of the nodeweights.} To be precise, we define the
\emph{$(m,s)$-discretization} ($m,s\in\Nbb$) as the map
$\xi_{m,s}:~X^+\to U^{s,m}$, where $\xi_{m,s}(x)$ is obtained by
considering the neighborhood $N_{\GG^+,s}(x)$, and replacing every
nodeweight $\xi(v)$ by $\lceil m \xi(v)\rceil$. Recall that the local
isomorphism $\phi:~V(\GG^+)\to V(\GG)$ is defined by $\phi:(x,\xi)
\mapsto x$.

\begin{lemma}\label{LEM:COLOR-APPROX}
For every $r\ge1$ and $\eps>0$, and every measurable $k$-coloring $c$
of $\GG^+$, there are positive integers $s,m$ and $l$, a measurable
$l$-coloring $h$ of $\GG$, and a map $f:~U^{s,m}\times[l]\to [k]$
such that the $k$-coloring $c'(x)= f\bigl(\xi_{s,m}(x),
h(\phi(x))\big)$ of $\GG^+$ satisfies
\[
d_{\rm var}\bigl(P_{\GG^+,r}[c],P_{\GG^+,r}[c']\bigr) \le \eps.
\]
\end{lemma}
\begin{proof}
Let $(X^+,\nu^+)$ be the underlying space of $\GG^+$. Let $\KK$
denote the set of all subsets of $X^+$ of the form
$\xi_{m,s}^{-1}(y)\cap \phi^{-1}(B)$, where $y\in U^{s,m}$, and $B$
is a Borel set of $X$. These sets generate the Borel sets of $X^+$,
hence by the Monotone Class Theorem, the closure under pointwise
convergence of the vector space generated by their indicator
functions contains every bounded Borel function on $X^+$.

In particular, there are pairs of integers $(m_i,s_i)$, colored balls
$y_i\in U^{s_i,m_i}$, Borel sets $B_i\subseteq X$ and real
coefficients $a_i$ $(i=1,\dots,N)$ such that
\[
\nu^+\Bigl\{x\in X^+:~\Bigl|c(x)-\sum_{i=1}^N
a_i\one(\xi_{m_i,s_i}(x)=y_i, \phi(x)\in B_i)\Bigr|\ge\frac12\Bigr\}
< \frac{\eps}{d^{r+1}}.
\]
Let $s=\max_i s_i$, $m=\prod_i m_i$, $l=2^N$, and let $h$ be a Borel
$l$-coloring of $X$ in which every $B_i$ is a union of color
classes. Then the sum in the above expression can be written as
$g\bigl(\xi_{s,m}(x),h(\phi(x))\bigr)$ for some $g:~U^{s,m}\times[l]\to
\mathbb{R}$. Rounding the values of $g$ to
the closest integer in $[k]$, we get a $k$-coloring $c'$ for which
\[
\nu\bigl\{x\in X^+:~c(x)\not=c'(x)\bigr\} <\frac{\eps}{d^{r+1}}.
\]
For a random point $x\in X^+$, the probability that the colorings $c$
and $c'$ differ on any node in its $r$-neighborhood is less than
$\eps$. This implies the lemma.
\end{proof}

Now we are able to prove the main theorem in this section.

\medskip

\noindent{\bf Proof of Theorem \ref{THM:BERLIFT}.} Our goal is to
approximate every element in $Q_{\GG^+,r,k}$ by an element in
$Q_{\GG,r,k}$ with arbitrary precision $\eps>0$. In other words, we
want to construct, for every measurable $k$-coloring $c$ of $\GG^+$,
a measurable $k$-coloring $c_0$ of $\GG$ that defines a similar
distribution of colored neighborhoods.

By Lemma~\ref{LEM:COLOR-APPROX} we may assume that $c$ is of the form
$f\bigl(\xi_{s,m}(x),h(\phi(x))\big)$ where $h$ is an $l$-coloring
of $\GG$ and $f:~U^{s,m}\to[k]$. Let $q$ be an $(s,\eps)$-quasirandom
$m$-coloring of $(\GG,h)$ guaranteed by Lemma~\ref{quasirand}, and
let $\GG'=(\GG,h \times q)$. Consider the $k$-coloring of $\GG$
defined by $c_0(z)=f(N_{\GG',s}(z),h(z))$. We claim that $c_0$ has
similar statistics as $c$:
\[
d_{\rm var}(P_{\GG^+,r}[c],P_{\GG,r}[c_0])\le\eps.
\]
This follows if we prove that the distributions of
$(\xi_{s,m}(y),h(\phi(y))$ (where $y$ is a random point of $\GG^+$)
and $(N_{\GG',s}(x),h(x))$ (where $x$ is a random point of $\GG$) are
close. But the distribution of $(\xi_{s,m}(y),h(\phi(y))$ is just
$\mu_{s,h,m}$, and the distribution of $(N_{\GG',s}(x),h(x))$ is
$\eps$-close to this by the quasirandomness of $q$. This completes
the proof.\proofend

\medskip

The following fact shows another connection between a graphing and
its associated Bernoulli graphing. We say that two graphings are
\emph{bi-locally isomorphic} if there exists a third graphing that
has local isomorphisms into both. The construction of the Bernoulli
lift implies that every graphing is bi-locally isomorphic to its
Bernoulli graphing. Since by the definition of the Bernoulli
graphing, two graphings are locally equivalent if and only if they
have the same Bernoulli graphing, we get the following more explicit
characterization:

\begin{proposition}\label{PROP:LOC-BILOC}
Two graphings are locally equivalent if and only if they are
bi-locally isomorphic.
\end{proposition}

To prove this proposition, it suffices to show that bi-local
isomorphism is a transitive relation. This takes some work which we
do not discuss here; for the details, we refer {the} reader to
\cite{LL}.

\medskip

Let us turn to graph sequences. {Every locally convergent graph
sequence determines a unique involution-invariant distribution and
through this, a Bernoulli graphing. One expects that among sequences
with the same local limit, a sequence with the least possible global
structure would converge to the Bernoulli graphing in the
local-global sense. As a special case, the following conjecture was
popularized by us in the past few years: {\it Let $G_n$ be a random
$d$-regular graph on $n$ vertices (if $d$ is odd, then we only
consider even values of $n$). Then $(G_n)_{n=1}^\infty$ is a
Bernoulli sequence with probability one.} In other words, the limit
object is the Bernoulli graphing produced from the $d$-regular tree.
A very recent paper of Gamarnik and Sudan \cite{GaSu} disproves this
conjecture.

The following weaker conjecture remains unsolved:}

\begin{conjecture}
A growing sequence of random $d$-regular graphs is local-global
convergent with probability one.
\end{conjecture}

{We don't know whether for $d\ge3$, the Bernoulli graphing
corresponding to the $d$-regular tree is the local-global limit of
any graph sequence.}

\section{Joins and maximal graphings}\label{SEC:JOINS}

{We show that every weak equivalence class of graphings contains a
maximal member. For this, we introduce a direct product-like
construction.

\begin{lemma}\label{LEM:JOIN}
Let $\GG,\GG_1,\GG_2,\dots$ be graphings and let $\phi_i:~V(\GG_i)\to
V(\GG)$ be local isomorphisms. Then there exists a graphing $\HH$ and
local isomorphisms $\psi_i:~V(\HH)\to V(\GG_i)$ and $\xi:~V(\HH)\to
V(\GG)$ such that $\phi_i\circ\psi_i=\xi$.
\end{lemma}

We call $\HH$ a {\it join} of the graphings $\GG_i$ relative to the
common ``factor'' $\GG$.

\medskip

\begin{proof}
Let $(X,\nu)$ be the underlying space of $\GG$, and let $(X_i,\nu_i)$
be the underlying space of $\GG_i$. First, consider the cartesian
product space $U=\prod_i X_i$. Let $\psi_i:~U\to X_i$ be the
coordinate maps, and consider the ``diagonal'' $\Delta=\{x\in
U:~\psi_i(\phi_i(x))=\psi_j(\phi_j(x))\text{for all }i,j\in\Nbb\}$.
By this definition, the map $\xi= \psi_i\circ\phi_i|_U$ is
independent of $i$.

We note that $\Delta$ is nonempty; in fact, $\xi(\Delta)$ has measure
$1$ in $\GG$. Indeed, the facts that $\phi_i$ is measure preserving
and the space $X_i$ is standard imply that $\phi_i(X_i)$ is a
measurable subset of $X$ of measure $1$. Hence so is the set
$W=\cap_i\phi_i(X_i)$. For any $x\in W$ and any choice
$y_i\in\phi_i^{-1}(x)$, we have $y=(y_i,y_2,\dots)\in U$ and
$\xi(y)=x$. The cartesian product graph $\HH'=\prod_i \GG_i$, defined
by
\[
V(\HH')=\prod_i V(\GG_i) \qquad
E(\HH')=\{(x,y):~(\psi_i(x),\psi_i(y))\in E(\GG_i)
~\text{for all}~i=1,2,\dots\},
\]
is not locally finite in general, but the induced subgraph
$\HH=\HH'[\Delta]$ is:

\begin{claim}\label{CLAIM:LOCFIN}
When restricted to any connected component of $\HH$, every coordinate
map $\psi_i$ gives an isomorphism between this connected component of
$\HH$ and a connected component of $\GG_i$. Consequently, all degrees
of $\HH$ are bounded by $d$.
\end{claim}

Let $x\in\Delta$, and consider the connected component $L$ of $\HH$
containing $x$, the connected component $J$ of $\GG$ containing
$\xi(x)$, and the connected component $J_i$ of $\GG_i$ containing
$\psi_i(x)$. The map $\phi_i$ is a local isomorphism, and hence it
gives an isomorphism between $J_i$ and $J$. Let $\zeta_i:~V(J)\to
V(J_i)$ be the inverse of this map, and define $\zeta(y)=
(\zeta_1(y),\zeta_2(y),\dots)$ for $y\in V(J)$. It is straightforward
to check that $\zeta$ is an embedding of $J$ into $\HH$, and that
there are no further edges of $\HH$ incident with the nodes of
$\zeta(V(J))$. Hence $\zeta(J)=L$. This proves the Claim.

We define a Polish space $Y$ on $\Delta$ by restricting the product
space $\prod_i X_i$ to $\Delta$. It is not hard to check that $\HH$
is a Borel graph on $Y$.

Next, we define a measure on $Y$. Let $A_i\subseteq X_i$ be Borel
sets so that only a finite number of them are proper subsets. Let
$\sigma_i(B)=\nu_i(A_i\cap\phi_i^{-1}(B))$ for every Borel subset
$B\subseteq X$, and consider the Radon-Nikodym derivative
$f_i=d\sigma_i/d\nu$. Define
\begin{equation}\label{EQ:MEASURE}
\mu(A_1\times A_2\times\cdots) = \int_X f_1f_2\dots\,d\nu.
\end{equation}
It is not hard to check that $\mu$ extends from these boxes to a
probability measure on all Borel sets in $\Delta$ (in ergodic theory,
this construction is called the {\it relatively independent joining}
of the measures $\nu_i$ over the common factor $\nu$; see e.g.
\cite{BCL}, Lemma 6.2 for a detailed description of this construction
for two factors). It is easy to see that every coordinate map
$\psi_i$ is measure preserving as a map from $(Y,\mu)\to(X,\nu)$.

\begin{claim}\label{CLAIM:MPRES}
The measure $\mu$, as a measure on the Borel graph $\HH$, is
involution invariant.
\end{claim}

To prove this, it suffices to construct a measure $\sigma^*$ on
$Y\times Y$ that is concentrated on $E(\HH)$ and
\begin{equation}\label{EQ:STAR}
\sigma^*(A\times B)=\int_A e_\HH(x,B)\,d\mu(x).
\end{equation}
Since the $\GG_i$ are graphings, we know that there are measures
$\eta_i^*$ on the Borels sets of $X_i\times X_i$, and $\eta^*$ on the
Borels sets of $X\times X$, related similarly to the measures $\nu_i$
and $\nu$. The space $Y\times Y$ is the cartesian product of the
spaces $X_i\times X_i$, and the maps $\phi_i$ define measure
preserving maps $\phi_i\times\phi_i:~(X_i\times
X_i,\eta_i^*)\to(X,\eta^*)$. We define a measure $\sigma^*$ similarly
to \eqref{EQ:MEASURE} above. It is easy to check that $\sigma^*$
satisfies \eqref{EQ:STAR} and it is concentrated on $E(\HH)$.

Thus we know that $\HH$ is a graphing, and the maps
$\psi_i:~V(\HH)\to V(\GG_i)$ and $\xi$ are local automorphisms.
\end{proof}

\begin{theorem}\label{THM:LOCGLOB-MAX}
In every local equivalence class $\CC$ of graphings there is a
largest one in the local-global partial order.
\end{theorem}

\begin{proof}
Let $Q_{r,k}$ denote the union of the sets $Q_{\GG,r,k}$, where
$\GG\in\CC$. There is a countable set of graphings
$\FF=\{\GG_1,\GG_2,\dots\}$ in the equivalence class such that
$\cup_i Q_{\GG_i,r,k}$ is dense in $Q_{r,k}$  for every $r$ and $k$.
It is enough to find a graphing that is larger than every Bernoulli
lift $\GG_i^+$ in the local-global partial order.

Let $\BB$ be the Bernoulli graphing in $\CC$. As shown in Section
\ref{SEC:BERN}, there are local isomorphisms $\phi_i:~V(\GG_i^+)\to
V(\BB)$. By Lemma \ref{LEM:JOIN}, there is a graphing $\HH$ and there
are local isomorphisms $\HH\to\GG_i^+$. This implies that $\HH$ is
above any of the $\GG_i^+$ in the local-global partial order.
\end{proof}
}

\section{Non-standard graphings}\label{NONSTAND}

{An alternative proof of Theorem \ref{thm:graphing} can be based on
the ultraproduct method of Elek and Szegedy \cite{ESz}.} Let
$(G_i)_{i=1}^\infty$ be an arbitrary graph sequence of maximum degree
at most $d$. Let $\omega$ be a non-principal ultrafilter on
$\mathbb{N}$. Let $\bG$ denote the ultraproduct of the graph
sequence. The vertex set $\bV$ of $\bG$ is the ultraproduct of the
vertex sets $V_i$ of $G_i$ and the edge set $\bE\subset \bV\times\bV$
is the {ultraproduct} of the edge sets $E_i\subset V_i\times V_i$ of
$G_i$. The graph $\bG$ has maximum degree at most $d$, since this
property is expressible by a first order formula. We can also
construct a $\sigma$-algebra $\AA$ on $\bV$ and a probability measure
$\mu$ on $\bV$ which is the ultralimit of the uniform distributions
on the sets $V_i$. It is not hard to check that $\bG$ satisfies the
graphing axiom \eqref{eq:degreeCond}.

If $(G_i)_{i=1}^\infty$ is a locally convergent graph sequence,
then $\bG$ has neighborhood frequencies that are the limits of the
neighborhood frequencies of the graphs $G_i$. If
$(G_i)_{i=1}^\infty$ is locally-globally convergent, then
$Q_{\bG,r,k}$ is the Hausdorff limit of the sets $Q_{G_i,r,k}$.

However, this does not directly prove Theorem \ref{thm:graphing},
since $(\bV,\mu)$ is not a separable probability space. One can
complete the proof by choosing an appropriate separable
sub-sigma-algebra of $\bG$ which preserves the graphing structure. We
omit the details here.

An attractive feature of ultralimit graphings is that the sets
$Q_{\bG,r,k}$ are all closed. It is not clear if there is a standard
graphing representation of the limit of a convergent sequence with
this stronger property.

\begin{question}
\label{ques:closedness}
 Let $(G_n)_{n=1}^\infty$ be a local-global convergent sequence of graphs. Is
there a graphing $\GG$ that represents the limit with the property
that $Q_{\GG,r,k}$ are all closed?
\end{question}

\section{Hyperfinite graphs and graphings}

For a graph $G$, we define $\tau_q(G)$ as the smallest $t$ such that
deleting $t$ appropriate nodes, every connected component of the
remaining graph has at most $q$ nodes. We say that a sequence
$(G_n)_{n=1}^\infty$ of finite graphs is
\emph{$(q,\eps)$-hyperfinite} if
$\liminf_n\tau_q(G_n)/|V(G_n)|\le\eps$. We say that
$(G_n)_{n=1}^\infty$ is \emph{hyperfinite} if for every $\eps>0$,
there is a $q$ such that $(G_n)_{n=1}^\infty$ is
$(q,\eps)$-hyperfinite. We can define hyperfiniteness of a graphing
$\GG$ on underlying space $X$ similarly: let $\tau_q(\GG)$ denote the
infimum of numbers $\delta \ge 0$ such that we can delete a Borel set
$S\subseteq X$ with measure $\delta$ so that every connected
component of the remaining graphing has at most $q$ nodes. We say
that a graphing $\GG$ is \emph{$(q,\eps)$-hyperfinite} if
$\tau_q(\GG)\le\eps$, and we say that $\GG$ is \emph{hyperfinite} if
for every $\eps>0$, there is a $q$ such that $\GG$ is
$(q,\eps)$-hyperfinite. Since we are talking about graphs with
bounded degree, we could replace deleting nodes by deleting edges in
the definitions of hyperfiniteness.

Hyperfiniteness in different settings was introduced by different
people (see Kechris and Miller \cite{KM}, Elek \cite{Elek1}, Schramm
\cite{Schramm}). Schramm proved that a  locally convergent sequence of
graphs is hyperfinite if and only if its limit is hyperfinite. This
does not hold for $(q,\eps)$-hyperfiniteness for a fixed pair $q$ and
$\eps$. As an easy example, a sequence of random $d$-regular graphs
tend to a limiting involution-invariant distribution (concentrated on
the infinite $d$-regular tree) that is $(1,1/2)$-hyperfinite, while the
sequence is not. On the other hand, a local-global convergent
sequence of graphs behaves nicer:

\begin{proposition}\label{PROP:HYPERFIN}
Let a sequence $(G_n)_{n=1}^\infty$ of finite graphs converge to a
graphing $\GG$ in the local-global sense. Then $(G_n)_{n=1}^\infty$
is $(q,\eps)$-hyperfinite if and only if $\GG$ is
$(q,\eps)$-hyperfinite.
\end{proposition}
\begin{proof}
A finite graph $G$ satisfies $\tau_q(G)\le \eps|V(G)|$ if and only
if it has a $2$-coloring $c$ such that
$P_{G,k,r}[c](c(\text{root})=1)\le \eps$ and  $P_{G,k,r}[c](B)=0$ for
every colored $r$-ball $B$ that contains a connected all-blue
subgraph with $k+1$ nodes. A graphing $\GG$ satisfies $\tau_q(\GG)\le
\eps$ if and only if for every $\eps'>\eps$, it has a $2$-coloring $c$
such that $P_{\GG,k,r}[c](c(\text{root})=1)\le \eps'$ and
$P_{\GG,k,r}[c](B)=0$ for every colored $r$-ball $B$ that contains a
connected all-blue subgraph with $k+1$ nodes. The proposition follows
by the definition of local-global convergence to a graphing.
\end{proof}

The following important property of hyperfiniteness is closely
related to the results of Schramm \cite{Schramm} and Benjamini,
Schramm and Shapira \cite{BSS}. It can be derived using the graph
partitioning algorithm of Hassidim, Kelner, Nguyen and Onak
\cite{HKNO}; a direct proof is given in \cite{LL}.

\begin{proposition}\label{PROP:HF-WEAKISO}
Hyperfiniteness is invariant under local equivalence.
\end{proposition}

Together with Proposition \ref{PROP:HYPERFIN}, this implies the above
mentioned result of Schramm that  a  locally convergent sequence of
graphs is hyperfinite if and only if its limit is hyperfinite. We
note that $(q,\eps)$-hyperfiniteness for a fixed $q$ and $\eps$ is
not invariant under local equivalence, which is shown, for example,
by the local-global limits of random $d$-regular graphs and of random
$d$-regular bipartite graphs. Our main result about hyperfinite
graphings is a strengthening of Corollary~\ref{cor:bermin}.

\begin{theorem}\label{THM:HYP-BERN}
Every atom-free hyperfinite graphing $\GG$ is locally-globally
equivalent to its Bernoulli graphing.
\end{theorem}

\begin{proof}
By Corollary~\ref{cor:bermin}, $\BB_\GG\prec\GG$. It remains to show
that $\GG\prec\BB_\GG$. In other words, for every coloring of $\GG$,
we have to find a coloring of $\BB_\GG$ with almost the same local
statistics.

Let $(X,\nu)$ be the underlying space of $\GG$, let $c:~X\rightarrow
[k]$ be a measurable coloring, and let us fix a radius
$r\in\mathbb{N}$ and an $\eps>0$. Let $\nu_B$ denote the measure of
$\BB_\GG$, and set $\eps_1=\eps/(8(d+1)^r)$. By Proposition
\ref{PROP:HF-WEAKISO}, the Bernoulli graphing $\BB_\GG$ of a
hyperfinite graphing $\GG$ is also hyperfinite.  Let
$S\subset\Gf[0,1]$ be a subset such that $\nu_B(S)\le\eps_1$ and
every connected component of $\Gf[0,1]\setminus S$ has at most $n$
nodes. Let $m\in\Nbb$, and define the coloring
$b:~\Gf[0,1]\rightarrow [m]\times\{0,1\}$ by $b(x)=(\lceil w(o)m
\rceil,\one_S(x))$ where $x=(G,o,w)$. Choosing $m$ large enough, we
may assume that the set $S'$ of points $x$ for which
$N_{\BB_\GG,r}(x)$ contains two points with the same color has
measure at most $\eps_1$. Note that $\nu_B(S \cup S') \le 2 \eps_1$
and all points of $\BB_\GG\setminus (S \cup S')$ are contained in
connected components that have at most $n$ vertices, and whose nodes
are colored differently by $b$.

On the other hand, by Corollary~\ref{cor:bermin} we have
$\BB_\GG\prec\GG$ which implies that there is a coloring
$b^*:~X\rightarrow [m]\times\{0,1\}$ such that
$$
d_\text{var}(P_{\GG,n}[b^*],P_{\BB_\GG,n}[b])\le \eps_1.
$$
It follows that there are subsets $T\subseteq
\Gf[0,1]$ and $T'\subseteq X$ with $\nu_B(T)=\nu(T')\le 4\eps_1$ such
that the following conditions hold:
\begin{itemize}
 \item[{\bf (a)}] All points of $\BB_\GG\setminus T$ are contained in
connected components that have at most $n$ vertices and whose nodes
are colored differently by $b$, and the same holds for the connected
components of $\GG \setminus T'$ with coloring $b^*$;

  \item[{\bf (b)}] Furthermore, for every $([m]\times\{0,1\})$-colored
connected graph $H$ with at most $n$ vertices, the measure of points
in components isomorphic to $H$ (as colored graphs) is the same in
$\BB_\GG\setminus T$ and $\GG\setminus T'$. Let $V_H(\BB_\GG
\setminus T)$ and $V_H(\GG \setminus T')$ be these two sets.

\end{itemize}

Let $C$ be a connected component of $\GG\setminus T'$. Since the
vertices of $C$ are colored differently by $b^*$, there is a (unique)
function $f_C:~[m]\times\{0,1\}\to [k]$ such that $c=f_C\circ b^*$ on
the nodes of $C$. This splits every set $V_H(\GG \setminus T')$ into
at most $k^{2m}$ measurable sets $V_{H,f}(\GG \setminus T')$ (indexed
by functions $f:~[m]\times\{0,1\}\to [k]$) that are unions of
components of $\GG\setminus T'$.

Split $V_H(\BB_\GG \setminus T)$ into sets $V_{H,f}(\BB_\GG \setminus
T)$ so that each $V_{H,f}(\BB_\GG \setminus T)$ is a union of
components of $\BB_\GG\setminus T$, and moreover
$\nu_B(V_{H,f}(\BB_\GG \setminus T))=\nu(V_{H,f}(\GG \setminus T'))$.
This is possible since there is no probability mass on any component
of $\BB_\GG$.

Let $c'$ be the  measurable $k$-coloring of $\BB_\GG$ defined in the
following way. Every $v \in V_{H,f}(\BB_\GG \setminus T)$ is colored
by $f \circ b(v)$, and the points in $T$ are all colored with one
arbitrary color in $[k]$. Note that the (conditional) local
statistics of $c'$ obtained by picking a random $v \in \BB_\GG$
conditioned on $N_{\BB_\GG,r}(v) \cap T = \emptyset$ is the same as
the (conditional) local statistics of $c$ obtained by picking a
random $v \in \GG$ conditioned on $N_{\GG,r}(v) \cap T' = \emptyset$.
The $\nu$-measure of the vertices $v \in \GG$ with $N_{\GG,r}(v) \cap
T' \neq \emptyset$ is at most $\nu(T') (d+1)^r \le 4 \eps_1 (d+1)^r$.
The same bound also holds for the $\nu_B$-measure of the vertices $v
\in \BB_\GG$ with $N_{\BB_\GG,r}(v) \cap T \neq \emptyset$. Thus we
have
\[
\dvar(P_{\GG,r}[c],P_{\BB_\GG,r}[c'])\le 8(d+1)^r\eps_1\le\eps
\]
which proves the theorem.
\end{proof}
\begin{remark}
As the proof of Theorem~\ref{THM:HYP-BERN} shows,
$\GG\prec\BB_\GG$ holds for every hyperfinite graphing $\GG$
(not necessarily atom-free).
\end{remark}

Now we are ready to state and prove our main theorem about
convergence of hyperfinite graph sequences. This theorem was proved
independently by Elek \cite{Elek2}.

\begin{theorem}
Every locally convergent hyperfinite graph sequence
$(G_n)_{n=1}^\infty$ with $|V(G_n)| \to \infty$ is a local-global
convergent Bernoulli sequence.
\end{theorem}

\begin{proof}
Let $(G_i)_{i=1}^\infty$ be a locally convergent hyperfinite
sequence, and let $\mu$ be the involution-invariant measure on $\Gf$
that is the local limit of the sequence. Since the Bernoulli graphing
$\mathcal{B}_\mu$ is locally equivalent to the local limit of
$(G_i)_{i=1}^\infty$, Proposition \ref{PROP:HF-WEAKISO} implies that
it is hyperfinite.

To prove the theorem, assume by contradiction that
$(G_i)_{i=1}^\infty$ does not converge in the local-global sense to
$\mathcal{B}_\mu$. Then it has a local-global convergent subsequence
whose limit graphing $\GG$ is not local-global equivalent to
$\BB_\GG=\mathcal{B}_\mu$. By Remark~\ref{rem:continuity} the
condition $|V(G_n)| \to \infty$ implies that $\GG$ is atom-free. This
however contradicts Theorem~\ref{THM:HYP-BERN}.
\end{proof}

\begin{corollary}\label{COR:LOC-LOCGLOB}
Local-global convergence is equivalent to local convergence when
restricted to growing hyperfinite graph sequences.
\end{corollary}

\section{Graphings as operators and expander graphings} \label{sec:expanders}

Let $\GG$ be a Borel graph on the probability space $(X,\mu)$ with
all degrees at most $d$. If $f:X\rightarrow\mathbb{C}$ is a
measurable function, then we define {$\GG f:~X\to \mathbb{C}$} by
$$\GG f(x)=\sum_{(x,v)\in E(\GG)}f(v).$$
It takes a short calculation to show that if $\GG$ is a graphing, then
it acts on the Hilbert space $L^2(X,\nu)$ as a bounded self-adjoint
operator. Let $f:X\rightarrow\mathbb{C}$ be an arbitrary function in
$L^2(X,\nu)$. Then
$$
\int_x |\GG f(x)|^2~d\nu\leq\int_x d\sum_{(x,v)\in
E(\GG)}|f(v)|^2~d\nu=d\int |f(x)|^2{\rm deg}(x)~d\nu\leq
d^2\|f\|_2^2.
$$
The equality in the above calculation uses the fact that $\GG$
satisfies (\ref{eq:degreeCond}). It is easy to see that
(\ref{eq:degreeCond}) is equivalent to the statement that the action
of $\GG$ is self-adjoint in the sense that $\langle \GG f,
g\rangle=\langle f,\GG g \rangle$ holds for every pair $f,g$ of
bounded measurable functions. This implies that the action of $\GG$
is also self-adjoint on $L^2(X,\nu)$. The Laplace operator
corresponding to a graphing is defined as $L=D-\GG$ where
$Df(x)=f(x){\rm deg}(x)$. It is easy to check that
\begin{equation}\label{laplace}
\langle Lf,f\rangle=\int_{(v,w)\in E(\GG)} (f(v)-f(w))^2~d\eta^*
\end{equation}
holds in $L^2(X,\nu)$ where $\eta^*$ is defined in
Section~\ref{graphing}. Thus $L$ is positive semidefinite .

\medskip

The theory of graphings is closely related to the theory of measure
preserving systems (in a sense, it generalizes ergodic theory). In
particular, one can define the notion of ergodicity. A graphing $\GG$
is \emph{ergodic} if there is no measurable partition of the vertex
set $X$ into positive measure sets $X_1,X_2$ such that there is no
edge between $X_1$ and $X_2$, or equivalently such that $X_1$ is a
union of connected components of $\GG$. Note that graphings, when
defined on an uncountable set, are never connected as graphs and so
the notion of ergodicity is a good replacement for the notion of
connectivity. Equation (\ref{laplace}) implies the following analogue
of a well known theorem from ergodic theory about the Koopman
representation (see \cite{Glas}).

\begin{proposition}\label{PROP:ERGODIC}
Let $L$ be the Laplace operator corresponding to the graphing $\GG$. The
multiplicity of the eigenvalue $0$ of $L$ as an operator on
$L^2(X,\nu)$ is $1$ if and only if $\GG$ is ergodic.
\end{proposition}

Graphings offer new phenomena. Ergodicity is equivalent to saying
that $\nu(N_1(S))>\nu(S)$ for every set $S$ with $0<\nu(S)\le 1/2$
(Here $N_1(S)=\cup_{x\in S} N_{\GG,1}(x)$). Positive expansion is a
natural strengthening of this condition. We say that a graphing $\GG$
is a \emph{$c$-expander} if for every Borel set $S\subseteq X$ with
$0<\nu(S)\le 1/2$, we have $\nu(N_1(S))\ge (1+c)\nu(S)$.  We say that
a graphing $\GG$ is an \emph{expander} if it is a $c$-expander for
some $c>0$.

{Let us restrict our attention to $d$-regular graphs and graphings.}
Let $(G_n)_{n=1}^\infty$ be a sequence of $d$-regular graphs that are
expanders with expansion $c>0$. Let us select a local-global
convergent subsequence. {It is easy to see that} its limit is a
$d$-regular graphing that is also a $c$-expander.

We can generalize spectral conditions for expanders to graphings. Let
us define \emph{spectral gap} of a $d$-regular graphing by
\[
\text{gap}(\GG)=\inf\{\langle Lf,f\rangle:~\langle f,f\rangle=1,
\langle f,1\rangle=0\}
\]
{(note that it does not matter whether we take the infimum over $f\in
L^2(X)$ or $f\in L^\infty(X)$)}. The following analogue of the
theorems of Alon and Milman \cite{AlMi} and Alon \cite{Alon} on
expanders can be proved along the same lines:

\begin{proposition}\label{PROP:EXPAND1}
Suppose that a $d$-regular graphing $\GG$ is a $c$-expander. Then
$c^2/(2d)\le \text{\rm gap}(\GG)\le 2c$. In particular, a graphing is
an expander if and only if its spectral gap is positive.
\end{proposition}

An easy calculation shows that if $\GG_1$ and $\GG_2$ are
local-global equivalent, then ${\rm gap}(\GG_1)={\rm gap}(\GG_2)$. In
other words ${\rm gap}(\GG)$ is a local-global invariant quantity.
This follows from the classical fact that measurable functions can be
arbitrarily well approximated by step functions. {It is also easy to
see that ${\rm gap}(\GG)$ is not invariant under local equivalence.}

One must be careful though: the spectral gap $\text{gap}(\GG)$ is a
lower bound on the eigenvalues of $\GG$ belonging to non-constant
eigenfunctions of $\GG$, but it may not be the infimum of such
eigenvalues. For example, the Bernoulli graphing of a 2-way infinite
path is ergodic but not an expander, and its Laplacian has no
non-constant eigenfunction.

\section{Graphings and local algorithms\label{sec:Alg}}

{\noindent{\bf Local algorithms and factor of i.i.d. processes.}}
Elek and Lippner \cite{ElLi} formulate a correspondence {principle}
between graphings and local algorithms. We can make this more precise
using the notion of Bernoulli graphings:

\medskip

\noindent\emph{Measurable graph theoretic statements for Bernoulli
graphings correspond to randomized local algorithms for finite
graphs.}

\medskip

Let us {consider} an example. Let $T$ be the $d$-regular tree with a
distinguished root and let $\Omega$ be the compact space
$[0,1]^{V(T)}$. Let $f:\Omega\rightarrow [k]$ be any measurable
function which depends only on the isomorphism class of the labeled
rooted tree. In other words $f$ is invariant under the action of the
root preserving automorphism group of $T$. Using the function $f$, we
create a random model of $k$ colorings of $T$ in the following way.
First we produce a random element $\omega\in\Omega$ by putting
independent random {weights} from $[0,1]$ on the vertices of $T$, and
then for every $v\in V(T)$, we define the color $c(v)$ as the value
of $f$ on the labeled rooted tree obtained from $T$ by assigning
labels $\omega$ and placing the root on $v$. We say that $f$ is the
\emph{rule} of the coloring process $c$. Such processes on the tree
are called \emph{factor of i.i.d. processes.} We say that the rule
$f$ has radius $r$ if it depends only on the labels on vertices of
$T$ that are of distance at most $r$ from the root.

The following rule (of radius {$1$}) is a classical method to
construct an independent set of nodes in a graph (see Alon and
Spencer \cite{AlSp}). Let $f:~\Omega\rightarrow\{0,1\}$ be the
function which returns $1$ if and only if the label on the root is
smaller than the labels on all the neighboring vertices. It is clear
that with probability one the corresponding random coloring $c$ is
the characteristic function of some independent set on $T$. We can
view $c$ as a randomized algorithm which produces an independent set
of points of density $1/(d+1)$. Since the rule $f$ has radius {$1$},
it can also be applied to a finite $d$-regular graph $G$. Let us put
random labels from $[0,1]$ on the vertices of $G$, and then evaluate
the rule $f$ at each vertex using only the neighborhood of radius
$1$. We get a random $\{0,1\}$ coloring of $V(G)$ such that $1$'s
form an independent set. Such algorithms (corresponding to a rule of
bounded radius) are called local algorithms. On the other hand, we
can view $f$ as the characteristic function of a single (non-random)
independent set in the Bernoulli graphing $\GG$ corresponding to the
tree $T$ (that is, $\GG:=\BB_\mu$ where $\mu$ is the {Dirac}
probability measure on the point $T \in \Gf$). The vertex set of
$\GG$ is $\Gf[0,1]$, but in $\GG$ almost every vertex is represented
by an element in $[0,1]^{V(T)}$, and so we can evaluate the function
$f$ for almost every point. It is clear now that $f^{-1}(1)$ is an
independent measurable set in $\GG$.

A general definition of factor of i.i.d. processes can be obtained
through Bernoulli graphings. Let $\mu$ be an involution-invariant
measure on $\Gf$, and let $\mathcal{B}_\mu$ be the corresponding
Bernoulli graphing on $\Gf[0,1]$. Let $f:~\Gf[0,1]\rightarrow [k]$ be
a Borel function. Then the involution-invariant measure
$\mu_{\mathcal{B},f}$ on $\Gf[k]$ has the property that it projects
to $\mu$ when the labels on the vertices are forgotten. In other
words $\mu_{\mathcal{B},f}$ puts a $k$-coloring process on the graphs
generated by $\mu$. The measure $\mu_{\mathcal{B},f}$ is called a
factor of i.i.d. process on $\mu$.  The rule of the process is the
function $f$. We say that the rule $f$ has radius $r$ if
$f(G_1)=f(G_2)$ whenever the balls of radius $r$ in $G_1$ and $G_2$
are isomorphic as rooted labeled graphs.

We can approximate the rule $f$ with an arbitrary precision $\eps$
with another rule $f'$ of finite radius $r$ (which depends on $\eps$)
in the sense that $\nu(x|f(x)\neq f'(x))\leq\eps$. An advantage of
the finite radius approximation is that it can be used for local
algorithms on finite graphs. Let $G$ be a finite graph of maximal
degree at most $d$, and let us put random labels from $[0,1]$ on the
vertices in $G$. Then $f'$ defines a new coloring of $G$ such that
the color of a vertex $v$ is computed using $f'$ for the labeled
neighborhood of radius $r$ of $v$.

\medskip

\noindent{\bf Nondeterministic property testing.} The connection
between the two convergence notions can be illuminated by the
following algorithmic considerations. Given a (very large) graph $G$
with bounded degree, we use the following sampling method to gain
information: we select randomly and uniformly a node of $G$, and
explore its neighborhood of radius $r$. We can repeat this $t$ times.
There are a number of algorithmic tasks (parameter estimation,
property testing) that can be studied in this framework; we only
sketch a simple version of property testing, and its connection to
local-global convergence.

It will be convenient to introduce the edit distance for graphs with
bounded degree. For two graphs on the same node set $V(G)=V(G')$, we
define
\[
d_1(G,G') = \frac{1}{|V(G)|}|E(G)\triangle E(G')|.
\]
For a graph property $\PP$, let
$\PP_{-\eps}=\{G\in\GG:~d_1(G,\PP)>\eps\}$.

We say that the graph property $\PP$ is \emph{testable} if for every
$\eps>0$, there are integers $r,t\ge 1$ such that given any graph $G$
that is large enough, taking $t$ samples of radius $r$ as described
above, we can guess whether the graph has property $\PP$: if
$G\in\PP$, then our guess should be ``YES'' with probability at least
$2/3$; if $G\in\PP_{-\eps}$, then the answer should be ``NO'' with
probability at least $2/3$. If  $\PP$ is testable, then a locally
convergent graph sequence cannot contain infinitely many graphs from
both $\PP$ and $\PP_{-\eps}$.

Now let us say that $\PP$ is \emph{nondeterministically testable} if
there is an integer $k\ge1$, and a testable property $\QQ$ of
$k$-colored graphs with bounded degree, such that $G\in\PP$ if and
only if there is a $k$-coloring $c$ such that $(G,c)\in\QQ$. This
$k$-coloring is a ``witness'' for our conclusion. As an example, the
property ``$G$ is the disjoint union of two graphs with at least
$|V(G)|/1000$ nodes'' is not testable, but it is nondeterministically
testable (a witness is a $2$-coloring with no edge between the $2$
colors); {so these two notions are different (in contrast to the case
of dense graphs \cite{LV}).} If $\PP$ is nondeterministically
testable, then a local-global convergent graph sequence cannot
contain infinitely many graphs from both $\PP$ and $\PP_{-\eps}$.

\section{Concluding remarks}\label{sec:concluding}

\noindent{\bf Local-global equivalence and limit representation.} We
have seen a characterization of local equivalence of two graphings
(Proposition \ref{PROP:LOC-BILOC}). Is there a similar
characterization of local-global equivalence?

Does every graphing represent the limit of a local-global convergent
graph sequence? This is stronger than the Aldous--Lyons conjecture,
but perhaps there is a counterexample. We can mention two possible
counterexamples suggested by our results.

Can a $d$-regular graphing be a better expander than any finite
$d$-regular graph? Such a graphing would certainly be a
counterexample. It is not easy, however, to compute the expansion
rate of even very simple graphings, like the Bernoulli tree.

Is every graphing $(d+1)$-edge-colorable in a Borel way? If a
graphing is the local-global limit of a sequence of finite {simple}
graphs, then these graphs can be $(d+1)$-edge-colored by Vizing's
Theorem, and it is not hard to see that such an edge-coloring can be
transferred to the limit graphing.

\medskip

\noindent{\bf Even finer limit notions.} Limit graphings can
represent even finer information than local-global convergence.
Consider the following examples. Let $0<a<1$ be an irrational number,
and consider the following three graphings: (a) $\CC_a$ is obtained
by connecting every point $x\in[0,1]$ to the two points $x\pm
a\pmod1$; (b) $\CC_a'$ consists of two disjoint copies of $\CC_a$
(both with measure $1/2$); (c) $\CC_a''$ is obtained by taking two
copies of $[0,1]$ (call them upper and lower), each with mass $1/2$,
and connecting every lower point $x\in[0,1]$ to the two upper points
$x\pm a\pmod1$.

These three graphings are locally isomorphic, and either one of them
represents the local-global limit of the sequence of cycles. But they
are ``different'': there is no measure preserving isomorphism between
them, and this has combinatorial reasons. The graphing $\CC_a'$ is
``disconnected'' (non-ergodic), while $\CC_a''$ is ``bipartite'': it
has a partition into two sets with positive measure such that every
edge connects the two classes. The graphing $\CC_a$ does not have any
partition with either one of these properties (even if we allow an
exceptional subset of measure $0$). This follows from basic ergodic
theory.

It seems that the graphing $\CC_a$ should represent the limit of odd
cycles, $\CC'_a$ should represent the limit of graphs consisting of a
pair of odd cycles, while $\CC''_a$ should represent the limit of
even cycles. This would correspond to a finer ordering of graphings,
where we say that say that a graphon $\GG_2$ is ``finer'' that a
graphing $\GG_2$ if $Q_{\GG_1,r,k}\subseteq Q_{\GG_2,r,k}$ for every
$r,k\ge1$. A theory of convergence that would explain these examples
has not been worked out, however.

We know \cite{BCKL} that local convergence is equivalent to
right-convergence where the target graph is in a small neighborhood
of the looped complete graph with all edge-weights $1$. Can
local-global convergence be characterized by, {or at least related
to,} some stronger form of right convergence?

\end{document}